\documentclass[10pt]{amsart}
\usepackage{amscd,amssymb}
\input xypic

\newtheorem{thm}{Theorem}[section]

\newtheorem{prop}[thm]{Proposition}

\newtheorem{rem}[thm]{Remark}

\def\square{\vbox{
      \hrule height 0.4pt
      \hbox{\vrule width 0.4pt height 5.5pt \kern 5.5pt \vrule width 0.4pt}
      \hrule height 0.4pt}}

\def\id{\mathrm{id}}

\def\ch\mathrm{c h}

\def\im{\mathrm{I m}}

\newcommand{\CoH}{\mathbf{CoH}}

\newcommand{\bfk}{\ensuremath{{\mathbf{k}}}}

\newcommand{\namedright}[3]{\ensuremath{#1\stackrel{#2}
 {\longrightarrow}#3}}
\newcommand{\nameddright}[5]{\ensuremath{#1\stackrel{#2}
 {\longrightarrow}#3\stackrel{#4}{\longrightarrow}#5}}
\newcommand{\namedddright}[7]{\ensuremath{#1\stackrel{#2}
 {\longrightarrow}#3\stackrel{#4}{\longrightarrow}#5
  \stackrel{#6}{\longrightarrow}#7}}

\newcommand{\larrow}{\relbar\!\!\relbar\!\!\rightarrow}
\newcommand{\llarrow}{\relbar\!\!\relbar\!\!\larrow}
\newcommand{\lllarrow}{\relbar\!\!\relbar\!\!\llarrow}

\newcommand{\lllnamedddright}[7]{\ensuremath{#1\stackrel{#2}
 {\lllarrow}#3\stackrel{#4}{\lllarrow}#5
  \stackrel{#6}{\lllarrow}#7}}

\newcommand{\qqed}{\hfill\Box}
\newcommand{\hlgy}[1]{\ensuremath{H_{*}(#1)}}
\newcommand{\rhlgy}[1]{\ensuremath{\widetilde{H}_{*}(#1)}} 
\newcommand{\cohlgy}[1]{\ensuremath{H^{*}(#1)}}  
\newcommand{\rcohlgy}[1]{\ensuremath{\widetilde{H}_{*}(#1)}}

\newcommand{\plocal}{\ensuremath{\mathbb{Z}_{(p)}}}
\newcommand{\cpinf}{\ensuremath{\mathbb{C}P^{\infty}}}  
\newcommand{\NSymm}{\ensuremath{\mbox{\textbf{NSymm}}}} 
\newcommand{\QSymm}{\ensuremath{\mbox{\textbf{QSymm}}}}

\numberwithin{equation}{section}

\begin{document}

\newcommand{\auths}[1]{\textrm{#1},}
\newcommand{\artTitle}[1]{\textsl{#1},}
\newcommand{\jTitle}[1]{\textrm{#1}}
\newcommand{\Vol}[1]{\textbf{#1}}
\newcommand{\Year}[1]{\textrm{(#1)}}
\newcommand{\Pages}[1]{\textrm{#1}}

\title{Decompositions of looped co-$H$-spaces and applications}  

\author{J. Grbi\'{c}$^{*}$} 
\address{School of Mathematics\\
University of Manchester\\
Manchester M13 9PL\\
United Kingdom}
\email{jelena.grbic@manchester.ac.uk}
\urladdr{http://www.maths.manchester.ac.uk/\~{}jelena}

\author{S. Theriault} 
\address{Department of Mathematical Sciences\\
University of Aberdeen\\
Aberdeen AB24 3UE\\
United Kingdom}
\email{s.theriault@maths.abdn.ac.uk}
\urladdr{http://www.maths.abdn.ac.uk/\~{}stephen}

\author{J. Wu$^{**}$} 
\address{Department of Mathematics\\
National University of Singapore\\
Singapore 119260\\
Republic of Singapore}\email{matwujie@math.nus.edu.sg}
\urladdr{http://www.math.nus.edu.sg/\~{}matwujie} 

\thanks{$^{*}$ Research supported in part by EPSRC grant xxx.} 
\thanks{$^{**}$ Research supported in part by the Academic Research Fund of the
National University of Singapore.}


\begin{abstract} 
We prove two homotopy decomposition theorems for the loops on 
co-$H$-spaces, including a generalization of the Hilton-Milnor 
Theorem. These are applied to problems arising in algebra, representation 
theory, toric topology, and the study of quasi-symmetric functions. 
\end{abstract}

\maketitle

\section{Introduction} 
\label{sec:intro} 

A central theme in mathematics is to decompose objects into 
products of simpler ones. The smaller pieces should then be 
simpler to analyze, and by understanding how the pieces are put 
back together information is obtained about the original object. 
In homotopy theory this takes the form of decomposing $H$-spaces 
as products of factors or decomposing co-$H$-spaces as wedges 
of summands. Powerful decomposition techniques have been 
developed. Some, such as those in~\cite{MNT,CMN} are 
concerned with decomposing specific spaces as finely as 
possible, while others, such as those in~\cite{SW1,STW2}, are 
concerned with functorial decompositions that are valid for all 
loop suspensions or looped co-$H$-spaces. 

In this paper we establish two new decomposition theorems 
that apply to looped co-$H$-spaces. One is a strong refinement 
of work in~\cite{STW2}, and the other is a generalized Hilton-Milnor 
Theorem. We give four applications which have connections with 
other areas of study in mathematics: the Poincar\'{e}-Birkhoff-Witt 
Theorem, Lie powers in representation theory, moment-angle 
complexes in toric topology, and quasi-symmetric functions. 

To state specific results, we introduce some notation and context. 
Let $p$ be an odd prime, and localize all spaces and 
maps at $p$. Take homology with mod-$p$ coefficients. 
Let $V$ be a graded module over $\mathbb{Z}/p\mathbb{Z}$ and let 
$T(V)$ be the tensor algebra on $V$. This tensor algebra is 
given a Hopf algebra structure by declaring that the generators 
are primitive and extending multiplicatively. In~\cite{SW1} it was 
shown that there is a coalgebra decomposition 
$T(V)\cong A^{\min}(V)\otimes B^{\max}(V)$ 
where $A^{\min}(V)$ is the minimal functorial coalgebra retract 
of $T(V)$ that contains $V$. One important property of this 
decomposition is that the primitive elements of $T(V)$ of tensor 
length not a power of~$p$ are all contained in the complement 
$B^{\max}(V)$. A programme of work ensued to geometrically 
realize these tensor algebra decompositions, which we now 
outline.  

By the Bott-Samelson theorem, there is an algebra isomorphism 
$\hlgy{\Omega\Sigma X}\cong T(\rhlgy{X})$. This was 
generalized in~\cite{Be} to the case of a simply-connected co-$H$ 
space~$Y$: there is an algebra isomorphism 
$\hlgy{\Omega Y}\cong T(\Sigma^{-1}\rhlgy{Y})$, where 
$\Sigma^{-1}\rhlgy{Y}$ is the desuspension by one degree of the 
graded module $\rhlgy{Y}$. Let $V=\Sigma^{-1}\rhlgy{Y}$ so 
$\hlgy{\Omega Y}\cong T(V)$. The coalgebra decomposition 
of $T(V)$ suggests that there are spaces $A^{\min}(Y)$ and $B^{\max}(Y)$ 
such that $\rhlgy{A^{\min}(Y)}\cong A^{\min}(V)$, 
$\rhlgy{B^{\max}(Y)}\cong B^{\max}(V)$, and there is a 
homotopy decomposition 
$\Omega Y\simeq A^{\min}(Y)\times B^{\max}(Y)$. Such 
a decompositions was realized in a succession of 
papers~\cite{SW1,SW2,STW1,STW2} which began with 
$Y$ being a $p$-torsion double suspension and ended 
with the general case of $Y$ being a simply-connected 
co-$H$-space. 

However, the story does not end there, as the module $B^{\max}(V)$ 
has a much richer structure. There is a coalgebra decomposition 
$B^{\max}(V)\cong T(\oplus_{n=2}^{\infty} Q_{n} B(V))$, 
where $Q_{n} B(V)$ is a functorial retract of $V^{\otimes n}$. 
Ideally, this should be geometrically realized as well. This 
was proved in~\cite{STW1} when $Y$ is a simply-connected, 
homotopy coassociative co-$H$ space. More precisely, there 
are spaces $Q_{n} B(Y)$ for $n\geq 2$ such that 
$\rhlgy{Q_{n} B(Y)}\cong\Sigma Q_{n} B(V)$, a homotopy 
fibration sequence 
\(\namedddright{\Omega Y}{\ast}{A^{\min}(Y)}{} 
          {\bigvee_{n=2}^{\infty} Q_{n} B(Y)}{}{Y}\), 
and a homotopy decomposition 
$\Omega Y\simeq A^{\min}(Y)\times 
    \Omega(\bigvee_{n=2}^{\infty} Q_{n} B(Y))$. 

In the more general case of a simply-connected co-$H$-space $Y$, 
the geometric realization of $A^{\min}(V)$ in~\cite{STW2} 
produced a homotopy decomposition 
$\Omega Y\simeq A^{\min}(Y)\times B^{\max}(Y)$ but it did 
not identify $B^{\max}(Y)$ as a loop space. The first goal of 
this paper is to do exactly that. 

\begin{thm} 
   \label{coHfib} 
   Let $Y$ be a simply-connected co-$H$-space and let 
   $V=\Sigma^{-1}\rhlgy{Y}$. There is a homotopy fibration sequence 
   \[\namedddright{\Omega Y}{}{A^{\min}(Y)}{} 
          {\bigvee_{n=2}^{\infty} Q_{n} B(Y)}{}{Y}\] 
   such that: 
   \begin{enumerate} 
      \item[1)] $\Omega Y\simeq A^{\min}(Y)\times 
                      \Omega(\bigvee_{n=2}^{\infty} Q_{n}B(Y))$; 
      \item[2)] $\rhlgy{A^{\min}(Y)}\cong A^{\min}(V)$; 
      \item[3)] for each $n\geq 2$, $\rhlgy{Q_{n} B(Y)}\cong\Sigma Q_{n} B(V)$.   
   \end{enumerate} 
\end{thm} 

In fact, Theorem~\ref{coHfib} is a special case of a more general 
theorem proved in Section~\ref{sec:geomreal} which geometrically 
realizes any natural coalgebra-split sub-Hopf algebra $B(V)$ of 
$T(V)$ as a loop space. 

The construction of the space $Q_{n} B(Y)$ exists by a suspension 
splitting result from~\cite{GTW}. To describe this, recall that 
James~\cite{J} proved that there is a homotopy decomposition 
$\Sigma\Omega\Sigma X\simeq\bigvee_{n=1}^{\infty}\Sigma X^{(n)}$, 
where $X^{(n)}$ is the $n$-fold smash of $X$ with itself. Note that  
$\rhlgy{\Sigma X^{(n)}}\cong\Sigma\rhlgy{X}^{\otimes n}$. 
James' decomposition was  generalized in~\cite{GTW}. If $Y$ is a 
simply-connected co-$H$-space then there is a homotopy decomposition 
$\Sigma\Omega Y\simeq\bigvee_{n=1}^{\infty} [\Sigma\Omega Y]_{n}$, 
where each space $[\Sigma\Omega Y]_{n}$ is a co-$H$-space 
and there is an isomorphism 
$\rhlgy{[\Sigma\Omega Y]_{n}}\cong\Sigma(\Sigma^{-1}\rhlgy{Y})^{\otimes n}$. 
Succinctly, $[\Omega\Sigma Y]_{n}$ is an $(n-1)$-fold 
desuspension of $Y^{(n)}$. A key point is that the space $Q_{n}B(Y)$ 
is a retract of the co-$H$-space $[\Sigma\Omega Y]_{n}$, so it too 
is a co-$H$-space. 

Our second result is a generalization of the Hilton-Milnor Theorem, 
touched upon in~\cite{GTW}. Recall that the Hilton-Milnor Theorem 
states that if $X_{1},\ldots,X_{m}$ are path-connected spaces then 
there is a homotopy decomposition 
\[\Omega(\Sigma X_{1}\vee\cdots\vee\Sigma X_{m})\simeq 
       \prod_{\alpha\in\mathcal{I}}\Omega(\Sigma X_{1}^{(\alpha_{1})} 
       \wedge\cdots\wedge X_{m}^{(\alpha_{m})})\] 
where $\mathcal{I}$ runs over a vector space basis of the free 
Lie algebra $L\langle x_{1},\cdots,x_{m}\rangle$, and if $w_{\alpha}$ 
is the basis element corresponding to $\alpha$ then $\alpha_{i}$ 
counts the number of occurances of $x_{i}$ in $w_{\alpha}$. 
Note that if $\alpha_{i}=0$ then, for example, we regard 
$X_{i}^{(\alpha_{i})}\wedge X_{j}^{(\alpha_{j})}$ as 
$X_{j}^{(\alpha_{j})}$ rather than as 
$\ast\wedge X^{(\alpha_{j})}\simeq\ast$. We generalize the 
Hilton-Milnor Theorem by replacing each $\Sigma X_{i}$ by a 
simply-connected co-$H$-space. 

\begin{thm} 
   \label{HM} 
   Let $Y_{1},\ldots, Y_{m}$ be simply-connected co-$H$-spaces. 
   There is a homotopy decomposition 
   \[\Omega(Y_{1}\vee\cdots\vee Y_{m})\simeq 
          \prod_{\alpha\in\mathcal{I}} 
          \Omega M((Y_{i},\alpha_{i})_{i=1}^{m})\]  
   where $\mathcal{I}$ runs over a vector space basis of the free 
   Lie algebra $L\langle y_{1},\ldots,y_{m}\rangle$ and: 
   \begin{enumerate} 
      \item[1)] each space $M((Y_{i},\alpha_{i})_{i=1}^{m})$ 
               is a simply-connected co-$H$-space; 
      \item[2)] $\rhlgy{M((Y_{i},\alpha_{i})_{i=1}^{m})}\cong 
                \Sigma\left((\Sigma^{-1}\rhlgy{Y_{1}})^{\otimes\alpha_{1}}\otimes 
                \cdots\otimes(\Sigma^{-1}\rhlgy{Y_{m}})^{\otimes\alpha_{m}}\right);$ 
      \item[3)] if $Y_{i}=\Sigma X_{i}$ for $1\leq i\leq m$ then 
                $M((Y_{i},\alpha_{i})_{i=1}^{m})\simeq 
                    \Sigma X^{(\alpha_{1})}\wedge\cdots\wedge X^{(\alpha_{m})}$.  
   \end{enumerate}  
\end{thm}  

Again, if $\alpha_{i}=0$ we interpret 
$(\Sigma^{-1}\rhlgy{Y_{i}})^{\otimes\alpha_{i}}\otimes 
    (\Sigma^{-1}\rhlgy{Y_{j}})^{\otimes\alpha_{j}})$ 
as $\Sigma^{-1}\rhlgy{Y_{j}})^{\alpha_{j}}$ rather than $0$. 
Note that Theorem~\ref{HM}~(3) is the usual Hilton-Milnor Theorem. 

Theorems~\ref{coHfib} and~\ref{HM} are very useful for producing 
homotopy decompositions of interesting spaces. In 
Section~\ref{sec:apps}, we give three examples: a complete 
decomposition of $\Omega Y$ into functorially indecomposable 
factors, a refined decomposition of some generalized moment-angle 
complexes that arise in toric topology, and a decomposition of 
$\Omega\Sigma\cpinf$ that implies a corresponding algebraic 
decomposition of the ring of quasi-symmetric functions.

\section{Geometric Realization of Natural Coalgebra-Split Sub-Hopf 
    Algebras} 
\label{sec:geomreal} 

In this section we prove Theorem~\ref{coHfib} as a special case of 
the more general Theorem~\ref{geomreal}. This gives conditions for 
when a sub-Hopf algebra of a tensor algebra has a geometric 
realization as a loop space. Before proving Theorem~\ref{geomreal} 
it will be useful to state two results. The first is a geometric realization 
statement from~\cite{STW2}. Recall that $p$ is an odd prime, the ground 
ring for all algebraic statements is $\mathbb{Z}/p\mathbb{Z}$, and all 
spaces and maps have been localized at $p$. 

\begin{thm} 
   \label{STWgeomreal} 
   Let $V$ be a graded module and suppose that $A(V)$ is a 
   functorial coalgebra retract of $T(V)$. Then $A(V)$ has a geometric 
   realization. That is, if $Y$ is a simply-connected co-$H$-space 
   such that there is an algebra isomorphism 
   $\hlgy{\Omega Y}\cong T(V)$, then there is a functorial 
   retract $\bar A(Y)$ of $\Omega Y$ with the property that 
   $\hlgy{\bar A(Y)}\cong A(V)$.~$\qqed$ 
\end{thm} 

Second, given a functorial coalgebra retract $A(V)$ of $T(V)$, 
let $A_{n}(V)$ be the component of $A(V)$ consisting of 
homogeneous elements of tensor length~$n$. The following 
suspension splitting theorem was proved in~\cite{GTW}. 

\begin{thm}
   \label{suspensionsplitting}
   Let $A(V)$ be any functorial coalgebra retract of $T(V)$
   and let $\bar A$ be the functorial geometric realization of
   $A$. Then for any simply-connected co-$H$-space~$Y$
   of finite type, there is a functorial homotopy decomposition
   \[\Sigma \bar A(Y)\simeq \bigvee_{n=1}^\infty \bar A_n(Y)\]
   such that $\bar{A}_{n}(Y)$ is a functorial retract of $[\Sigma\Omega Y]_{n}$ 
   and there is a coalgebra isomorphism 
   \[\widetilde H_*(\bar A_n(Y))\cong
       A_n(\Sigma^{-1}\widetilde H_*(Y))\]
   for each $n\geq 1$.
\end{thm}

Now suppose that $B(V)$ is a sub-Hopf algebra of $T(V)$. We say 
that $B(V)$ is \emph{coalgebra-split} if the inclusion 
\(\namedright{B(V)}{}{T(V)}\) 
has a natural coalgebra retraction. Observe that the 
weaker property of $B(V)$ being a sub-coalgebra of $T(V)$ 
which splits off $T(V)$ implies by Theorem~\ref{STWgeomreal} 
that $B(V)$ has a geometric realization $\overline{B}$. We aim to 
show that the full force of $B(V)$ being a sub-Hopf algebra of $T(V)$ 
implies that it has a much more structured geometric realization.   

If $M$ is a Hopf algebra, 
let $QM$ be the set of indecomposable elements of $M$, and 
let $IM$ be the augmentation ideal of $M$. If $B(V)$ is a natural 
sub-Hopf algebra of $T(V)$ then there is a natural epimorphism 
\(\namedright{IB(V)}{}{QB(V)}\). 
Let $T_{n}(V)$ be the component of $T(V)$ consisting of the 
homogeneous tensor elements of length~$n$, and let 
$B_{n}(V)=IB(V)\cap T_{n}(V)$. Let $Q_{n}B(V)$ be the quotient 
of $B_{n}(V)$ in $QB(V)$. Let $\CoH$ be the category of 
simply-connected co-$H$-spaces and co-$H$-maps and 
let $\bfk=\mathbb{Z}/p\mathbb{Z}$. 

\begin{thm} 
  \label{geomreal}
   Let $B(V)$ be a natural coalgebra-split sub-Hopf algebra of $T(V)$  
   and let $\bar B$ be its geometric realization. Then there exist functors 
   $\bar Q_nB$ from $\CoH$ to spaces such that for any $Y\in \CoH$: 
   \begin{enumerate}
      \item[1)] $\bar Q_nB(Y)$ is functorial retract of $[\Sigma\Omega Y]_n$;
      \item[2)] there is a functorial coalgebra isomorphism
                    \[\Sigma^{-1}\bar H_*(\bar Q_nB(Y))\cong Q_nB(\Sigma^{-1}\bar H_*(Y));\]   
      \item[3)] there is a natural homotopy equivalence
                     \[\bar B(Y)\simeq \Omega \left(\bigvee_{n=1}^\infty \bar Q_nB(Y)\right).\] 
   \end{enumerate}
\end{thm} 

\begin{proof}
The proof is to give a geometric construction for the
indecomposables of $B(V)$. Let $B^{[n]}(V)$ be the sub-Hopf algebra
generated by $Q_iB(V)$ for $i\leq n$. By the method of proof 
of~\cite[Theorem 1.1]{LLW}, each $B^{[n]}(V)$ is a natural coalgebra-split 
sub-Hopf algebra of $T(V)$, and there is a natural coalgebra decomposition
\begin{equation} 
  \label{BnVsplit} 
  B^{[n]}(V)\cong B^{[n-1]}(V)\otimes A^{[n]}(V), 
\end{equation}  
where $A^{[n]}(V)=\bfk\otimes_{B^{[n-1]}(V)}B^{[n]}(V)$. Note that
\[Q_nB(V)\cong A^{[n]}(V)_n.\] 

By Theorem~\ref{STWgeomreal}, the functorial coalgebra splitting 
in~(\ref{BnVsplit}) has a geometric realization as a natural homotopy 
decomposition
\begin{equation} 
  \label{geomreal1}
  \bar B^{[n]}(Y)\simeq \bar B^{[n-1]}(Y)\times \bar A^{[n]}(Y)
\end{equation}
for some $Y\in \CoH$. This induces a filtered decomposition with
respect to the augmentation ideal filtration of $H_*(\Omega Y)$. By
Theorem~\ref{suspensionsplitting},
\[\Sigma\bar A^{[n]}(Y)\simeq \bigvee_{k=1}^\infty \bar A^{[n]}_k(Y)\] 
where $\bar{A}^{[n]}_{k}(Y)$ is a functorial retract of $[\Sigma\Omega Y]_{n}$ 
and $\bar A^{[n]}_k(Y)\simeq\ast$ for $k<n$ because $A^{[n]}_k(V)=0$ 
for $0<k<n$. Define
\[\bar Q_nB(Y)=\bar A^{[n]}_n(Y).\] 
Let $\phi_n$ be the composite of inclusions 
\begin{equation} 
  \label{phindef}
  \begin{array}{ccl}
     \bar Q_nB(Y)=\bar A^{[n]}_n(Y)& \longrightarrow & \Sigma \bar A^{[n]}(Y)\\
     & \longrightarrow & \Sigma ( \bar B^{[n-1]}(Y)\times \bar A^{[n]}(Y))\\
     & \stackrel{\simeq}{\longrightarrow} & \Sigma \bar B(Y)\\
     & \longrightarrow &\Sigma \Omega Y.\\
  \end{array}
\end{equation} 

Consider the composite
\begin{equation} 
  \label{phincomp}
     \lllnamedddright{\Omega\left(\bigvee_{n=1}^\infty\bar
     Q_nB(Y)\right)}{\Omega(\bigvee_{n=1}^\infty\phi_n)}
     {\Omega\Sigma\Omega Y}{\Omega\sigma}{\Omega Y}{r}{\bar B(Y)},
\end{equation}
where $\sigma$ is the evaluation map and $r$ is the retraction map. 
We wish to show that this composite induces an isomorphism in 
homology, implying that it is a homotopy equivalence. The assertions 
of the theorem would then follow. To show that~(\ref{phincomp}) induces 
an isomorphism in homology it suffices to filter appropriately and show 
that we obtain an isomorphism of associated graded objects. 

Let
\[H_*(\Omega\Sigma\Omega Y)=T(\bar H_*(\Omega Y))\] 
be filtered by
\[I^nH_*(\Omega\Sigma\Omega Y)=\sum_{t_1r_1+\cdots + t_sr_s\geq n}
  (I^{t_1}H_*(\Omega Y))^{\otimes r_1}\otimes\cdots\otimes
  (I^{t_s}H_*(\Omega Y))^{\otimes r_s}.\] 
Filter $\hlgy{\Omega Y}$ by the augmentation ideal filtration. Then
\[\namedright{\Omega\sigma_*\colon H_*(\Omega\Sigma\Omega Y)} 
     {}{H_*(\Omega Y)}\] 
is a filtered map since $\Omega\sigma_*$ is an algebra map. 
Let $H_*(\bar B(Y))$ be filtered subject
to the augmentation ideal filtration of $H_*(\Omega Y)$. Then $r_*$
is a filtered map. Note that as an algebra
\[H_*\left(\Omega\left(\bigvee_{n=1}^\infty\bar
    Q_nB(Y)\right)\right)=T\left(\bigoplus_{n=1}^\infty \Sigma^{-1}(\bar
    H_*(\bar Q_nB(Y)))\right),\] 
which is filtered by
\[\sum_{i_1+\cdots+i_t\geq n} \Sigma^{-1}(\bar H_*(\bar  
  Q_{i_1}B(Y)))\otimes\cdots\otimes \Sigma^{-1}(\bar H_*(\bar
  Q_{i_t}B(Y))).\]
Observe that $\phi_{n\ast}$ maps $\bar H_*(\bar Q_nB(Y))$ into
$\Sigma I^nH_*(\Omega Y)$ and the composite
\begin{equation} 
  \label{phinhlgycomp}
     \namedright{\bar H_*(\bar Q_nB(Y))}{\phi_{n\ast}}
     {\Sigma I^nH_*(\Omega Y)}\twoheadrightarrow
     \Sigma I^nH_*(\Omega Y)/\Sigma I^{n+1}H_*(\Omega Y)
\end{equation}
is a monomorphism because $\bar Q_nB(Y)$ is obtained from the
$n$-homogenous component of $\Sigma \bar A^{[n]}(Y)$. Thus
$\Omega(\bigvee_{n=1}^\infty\phi_n)_*$ is a filtered map and the
image of
\[E^0(\Omega\sigma_*\circ \Omega(\bigvee_{n=1}^\infty\phi_n)_*)\] 
is the sub-Hopf algebra of 
$E^0H_*(\Omega Y)=T(\Sigma^{-1}\bar H_*(Y))$ 
generated by
\[E^0\phi_{n\ast}(\Sigma^{-1}\bar H_*(\bar Q_nB(Y)))\] 
for $n\geq 1$. From~(\ref{phinhlgycomp}),
\[\Sigma^{-1}\bar H_*(\bar Q_nB(Y))\cong
     E^0\phi_{n\ast}(\Sigma^{-1}\bar H_*(\bar Q_nB(Y))).\] 
By the construction of $\phi_n$ in~(\ref{phindef}), the modules
\[\{E^0\phi_{n\ast}(\Sigma^{-1}\bar H_*(\bar Q_nB(Y)))\}\] 
are algebraically independent because $\bar Q_iB(Y)$ is mapped into
$\Sigma B^{[n]}(Y)$ for $i\leq n$ and $\bar Q^{[n]}(Y)$ is mapped
into $\Sigma \bar A^{[n]}(Y)$ which is the complement to 
$\Sigma\bar B^{[n-1]}(Y)$. Since each $\bar Q_nB(Y)$ is mapped 
into $\Sigma\bar B(Y)$,
\[\im(E^0(\Omega\sigma_*\circ \Omega(\bigvee_{n=1}^\infty\phi_n)_*))
    =T(E^0\phi_{n\ast}(\Sigma^{-1}\bar H_*(\bar Q_nB(Y))))\] 
is a sub-Hopf algebra of 
$E^0H_*(\bar B(Y))\subseteq T(\Sigma^{-1}\bar H_*(Y))$. 
By computing the Poincar\'e series,
\[\im(E^0(\Omega\sigma_*\circ
\Omega(\bigvee_{n=1}^\infty\phi_n)_*))=E^0H_*(\bar B(Y)).\] 
Since
\(r\colon\namedright{\Omega Y}{}{\bar B(Y)}\)
is a retraction map,
\[E^0r_*|_{E^0H_*(\bar B(Y))}=\id_{E^0H_*(\bar B(Y))}.\] 
Therefore the composite
\[E^0r_*\circ E^0(\Omega\sigma_*)\circ E^0
    \Omega(\bigvee_{n=1}^\infty\phi_n)_*\] 
of associated graded objects induced by the composition 
in~(\ref{phincomp}) is an isomorphism, as required.  
\end{proof}

The proof of Theorem~\ref{geomreal} does more. Recall the map 
\(\namedright{\bar Q_{n}B(Y)}{}{\Sigma\Omega Y}\) 
defined in~(\ref{phindef}). Taking the wedge sum for $n\geq 1$ 
and then evaluating, we obtain a composite 
\[\phi\colon\nameddright{\bigvee_{n=1}^\infty \bar
    Q_nB(Y)}{\phi_n}{\Sigma\Omega Y}{\sigma}{Y}.\] 
The thrust of the proof of Theorem~\ref{geomreal} was to show 
that the composite in~(\ref{phincomp})  is a homotopy equivalence. 
That is, the composite $r\circ\Omega\phi$ is a homotopy equivalence. 
In particular, this implies that $\Omega\phi$ has a functorial 
retraction. Consequently, if $\bar{A}(Y)$ is the homotopy fiber of $\phi$ 
we immediately obtain the following. 

\begin{thm} 
   \label{geomrealfib} 
   Let $B(V)$ be a natural coalgebra-split sub-Hopf algebra of $T(V)$
   and let the functor $A$ be given by $A(V)=\bfk\otimes_{B(V)}T(V)$.
   Then there is a homotopy fibration sequence 
   \[\nameddright{\Omega \left(\bigvee_{n=1}^\infty \bar
      Q_nB(Y)\right)}{\Omega\phi}{\Omega Y}{}{\bar A(Y)}
      \longrightarrow\namedright{\bigvee_{n=1}^\infty \bar Q_nB(Y)}{\phi}{Y}\] 
   where $Y\in\CoH$ and a functorial decomposition 
   \[\Omega Y\simeq \Omega \left(\bigvee_{n=1}^\infty \bar
         Q_nB(Y)\right)\times \bar A(Y).\] 
   $\qqed$ 
\end{thm} 

Note that $\bar A$ is a geometric realization of $A$. 
\medskip 

\noindent 
\begin{proof}[Proof of Theorem~\ref{coHfib}] 
In Theorem~\ref{geomrealfib} we can choose $B(V)$ to be $B^{\max}(V)$. 
The fiber $\bar{A}(Y)$ of $\phi$ is now, by definition, $A^{\min}(Y)$. The 
theorem follows immediately. 
\end{proof}

\section{The generalization of the Hilton-Milnor Theorem} 
\label{sec:HM} 

In this section we prove Theorem~\ref{HM}. We begin by stating 
a key general result from~\cite{GTW}. 

\begin{thm} 
   \label{GTW} 
   Let $Y$ and $Z$ be simply-connected co-$H$-spaces. There 
   is a homotopy decomposition 
   \[Z\wedge\Omega Y\simeq\bigvee_{n=1}^{\infty} [Z\wedge\Omega Y]_{n}\] 
   such that: 
   \begin{enumerate} 
      \item[1)] each space $[Z\wedge\Omega Y]_{n}$ is a simply-connected 
                     co-$H$-space; 
      \item[2)] $\rhlgy{[Z\wedge\Omega Y]_{n}}\cong 
                           \rhlgy{Z}\otimes(\Sigma^{-1}\rhlgy{Y})^{\otimes n}$; 
      \item[3)] if $Z=S^{1}$ and $Y=\Sigma X$ then 
                     $[Z\wedge\Omega\Sigma X]_{n}\simeq\Sigma X^{(n)}$. 
   \end{enumerate} 
   $\qqed$ 
\end{thm} 

In particular, if $Z=S^{1}$ then we obtain a homotopy decomposition 
of $\Sigma\Omega Y$ which generalizes James' decomposition of 
$\Sigma\Omega\Sigma X$, as discussed in the Introduction. The 
application of Theorem~\ref{GTW} that we need is the following. 

\begin{prop} 
   \label{GTWprop} 
   Let $Y_{1},\ldots,Y_{m}$ be simply-connected co-$H$-spaces. 
   There is a homotopy decomposition 
   \[\Sigma\Omega Y_{1}\wedge\cdots\wedge\Omega Y_{m}\simeq 
           \bigvee_{n_{1},\ldots,n_{m}=1}^{\infty}  M((Y_{i},n_{i})_{i=1}^{m})\] 
   such that: 
   \begin{enumerate} 
      \item[1)] each space $M((Y_{i},n_{i})_{i=1}^{m})$ 
                    is a simply-connected co-$H$-space; 
      \item[2)] $\rhlgy{M((Y_{i},n_{i})_{i=1}^{m})}\cong 
                            \Sigma\left((\Sigma^{-1}\rhlgy{Y_{1}})^{\otimes n_{1}}\otimes\cdots  
                                 \otimes(\Sigma^{-1}\rhlgy{Y_{m}})^{\otimes n_{m}}\right)$; 
      \item[3)] if $Y_{i}=\Sigma X_{i}$ for $1\leq i\leq m$ then 
                    $M((\Sigma X_{i},n_{i})_{i=1}^{m})\simeq 
                         \Sigma X^{(n_{1})}\wedge\cdots\wedge X^{(n_{m})}$.  
   \end{enumerate} 
\end{prop} 

\begin{proof} 
First, consider the special case when $m=1$. We wish to decompose 
$\Sigma\Omega Y_{1}$. Applying Theorem~\ref{GTW} with $Z=S^{1}$ 
and $Y=Y_{1}$, we obtain a homotopy decomposition 
\[\Sigma\Omega Y\simeq\bigvee_{n_{1}=1}^{\infty} M(Y_{1},n_{1})\] 
where $M(Y_{1},n_{1})=[\Sigma\Omega Y_{1}]_{n_{1}}$. In particular,  
$M(Y_{1},n_{1})$ is a simply-connected co-$H$-space,  
$\rhlgy{M(Y_{1},n_{1})}\cong 
        \Sigma(\Sigma^{-1}\rhlgy{Y_{1}})^{\otimes n_{1}}$, 
and if $Y_{1}=\Sigma X_{1}$ then 
$M(\Sigma X_{1},n_{1})\simeq\Sigma X_{1}^{(n_{1})}$. 

Next, consider the special case when $m=2$. We wish to decompose 
$\Sigma\Omega Y_{1}\wedge\Omega Y_{2}$. From the $m=1$ case we have 
\[\Sigma\Omega Y_{1}\wedge\Omega Y_{2}\simeq 
      \left(\bigvee_{n_{1}=1}^{\infty} M(Y_{1},n_{1})\right)\wedge\Omega Y_{2}\simeq 
      \bigvee_{n_{1}=1}^{\infty} M(Y_{1},n_{1})\wedge\Omega Y_{2}.\] 
Since $M(Y_{1},n_{1})$ is a co-$H$-space, for 
each $n_{1}\geq 1$ we can apply Theorem~\ref{GTW} with 
$Z=M(Y_{1},n_{1})$ and $Y=Y_{2}$ to further decompose 
$M(Y_{1},n_{1})\wedge\Omega Y_{2}$. Collecting these, we 
obtain a homotopy decomposition 
\[\Sigma\Omega Y_{1}\wedge\Omega Y_{2}\simeq 
       \bigvee_{n_{1},n_{2}=1}^{\infty} M((Y_{i},n_{i})_{i=1}^{2})\] 
where each space $M((Y_{i},n_{i})_{i=1}^{2})$ is a simply-connected 
co-$H$-space,  
\[\rhlgy{M((Y_{i},n_{i})_{i=1}^{m})}\cong\Sigma\left( 
         (\Sigma^{-1}\rhlgy{Y_{1}})^{\otimes n_{1}}\otimes 
         (\Sigma^{-1}\rhlgy{Y_{2}})^{\otimes n_{2}}\right)\]   
and if $Y_{i}=\Sigma X_{i}$ then 
$M((Y_{i},n_{i})_{i=1}^{2})\simeq\Sigma X_{1}^{(n_{1})}\wedge X^{(n_{2})}$. 

More generally, if $m>2$ then the procedure in the previous 
paragraph is iterated to obtain the homotopy decomposition asserted 
in the statement of the proposition. 
\end{proof} 

As a final preliminary result, we state a homotopy decomposition 
proved in~\cite{P}. For a space $X$ and an integer $j$, let 
$j\cdot X=\bigvee_{i=1}^{j} X$. 

\begin{thm} 
   \label{Porter} 
   Let $X_{1},\ldots,X_{m}$ be simply-connected $CW$-complexes of 
   finite type. Let $F$ be the homotopy fiber of the inclusion 
   \(\namedright{\bigvee_{i=1}^{m} X_{i}}{}{\prod_{i=1}^{m} X_{i}}\). 
   There is a homotopy equivalence 
   \[F\simeq\bigvee_{j=2}^{m}\left(
     \bigvee_{1\leq i_{1}<\cdots<i_{j}\leq m} (j-1)\cdot 
     \Sigma\Omega X_{i_{1}}\wedge\cdots\wedge\Omega X_{i_{j}}\right).\] 
   $\qqed$ 
\end{thm}  

\begin{rem} 
\label{Porterremark} 
A version of Theorem~\ref{Porter} holds for an infinite wedge 
$\bigvee_{i=1}^{\infty} X_{i}$, provided the spaces $X_{i}$ can 
be ordered so that the connectivity of $X_{i}$ is nondecreasing and 
tends to infinity. This guarantees that the fiber $F$ 
of the inclusion 
\(\namedright{\bigvee_{i=1}^{\infty} X_{i}}{}{\prod_{i=1}^{\infty} X_{i}}\) 
is of finite type. 
\end{rem} 

\noindent 
\begin{proof}[Proof of Theorem~\ref{HM}] 
We first consider the special case when $Y_{i}=\Sigma X_{i}$, that is, 
the usual Hilton-Milnor Theorem. One way to think of the proof is  
as follows. First, including the wedge into the product gives a 
homotopy fibration 
\[\nameddright{F_{1}}{}{\bigvee_{i=1}^{m}\Sigma X_{i}}{} 
     {\prod_{i=1}^{m}\Sigma X_{i}}\] 
that defines the space $F_{1}$. This fibration splits after looping as 
\begin{equation} 
  \label{HMstep1} 
  \Omega(\bigvee_{i=1}^{m}\Sigma X_{i})\simeq 
     \prod_{i=1}^{m}\Omega\Sigma X_{i}\times\Omega F_{1}. 
\end{equation}  
Second, by Theorem~\ref{Porter}, 
\[F_{1}\simeq\bigvee_{j=2}^{m}\left(
     \bigvee_{1\leq i_{1}<\cdots<i_{j}\leq m} (j-1)\cdot 
     \Sigma\Omega\Sigma X_{i_{1}}\wedge\cdots\wedge\Omega\Sigma X_{i_{j}}\right).\]
Iteratively using James' decomposition 
$\Sigma\Omega\Sigma X\simeq\bigvee_{n=1}^{\infty}\Sigma X^{(n)}$ 
we obtain a refined decomposition 
\[F_{1}\simeq\bigvee_{\alpha_{1}\in\mathcal{J}_{1}} M_{\alpha_{1}}\] 
for some index set $\mathcal{J}_{1}$, where each $M_{\alpha_{1}}$ is of 
the form $\Sigma X_{t_{1}}^{(r_{1})}\wedge\cdots\wedge X_{t_{l}}^{(r_{l})}$ 
for $l\geq 2$, $1\leq t_{1}<\cdots<t_{l}\leq m$ and $r_{1},\ldots,r_{l}\geq 1$. 
Third, including the wedge into the product gives a homotopy fibration 
\[\nameddright{F_{2}}{}{\bigvee_{\alpha_{1}\in\mathcal{J}_{1}} M_{\alpha_{1}}} 
         {}{\prod_{\alpha_{1}\in\mathcal{J}_{1}} M_{\alpha_{1}}}\] 
which defines the space $F_{2}$. This fibration splits after looping 
so~(\ref{HMstep1}) refines to a decomposition 
\begin{equation} 
  \label{HMstep2} 
  \Omega(\bigvee_{i=1}^{m}\Sigma X_{i})\simeq\prod_{i=1}^{m}\Omega\Sigma X_{i} 
       \times\prod_{\alpha_{1}\in\mathcal{J}_{1}}\Omega M_{\alpha_{1}}\times 
       \Omega F_{2}. 
\end{equation}  
Observe that since each $\Sigma X_{i}$ is simply-connected, the spaces 
$M_{\alpha_{1}}$ can be ordered so their connectivity is nondecreasing 
and tending to infinity. Therefore, Remark~\ref{Porterremark} 
implies that Theorem~\ref{Porter} can be applied to decompose $F_{2}$. 
The process can now be iterated to produce 
fibers $F_{k}$ for $k\geq 3$ and a decomposition  
\begin{equation} 
 \label{HMstepk} 
  \Omega(\bigvee_{i=1}^{m}\Sigma X_{i})\simeq\prod_{i=1}^{m}\Omega\Sigma X_{i} 
       \times\left(\prod_{j=1}^{k-1}\prod_{\alpha_{j}\in\mathcal{J}_{j}}\Omega M_{\alpha_{j}} 
       \right)\times\Omega F_{k}. 
\end{equation}  
where each $M_{\alpha_{j}}$ is of the form 
$\Sigma X_{t_{1}}^{(r_{1})}\wedge\cdots\wedge X_{t_{l}}^{(r_{l})}$ 
for $l\geq j+1$, $1\leq t_{1}<\cdots<t_{l}\leq m$ and $r_{1},\ldots,r_{l}\geq 1$. 
Note that the condition $l\geq j+1$ implies that the connectivity of $F_{k}$ 
is strictly increasing with $k$, and so tends to infinity. Thus, the 
decompositions of $\Omega(\bigvee_{i=1}^{m}\Sigma X_{i})$ stabilize. 
What remains is a bookkeeping argument that makes explicit the 
factors $\Sigma X_{t_{1}}^{(r_{1})}\wedge\cdots\wedge X_{t_{l}}^{(r_{l})}$. 
As stated in the Introduction, this takes the form of an index set 
determined by a vector space basis of the free Lie algebra 
$L\langle x_{1},\ldots,x_{m}\rangle$. 

We wish to generalize this to the case of $\Omega(\bigvee_{i=1}^{m} Y_{i})$ 
for simply-connected co-$H$-spaces $Y_{i}$. To do so, simply 
replace the use of James' decomposition above with Proposition~\ref{GTW}. 
The rest of the argument goes through verbatim. 
\end{proof}

\section{Applications} 
\label{sec:apps} 

In this section we give four examples to illustrate how Theorems~\ref{coHfib} 
and~\ref{HM} can be used to produce useful homotopy decompositions 
of spaces of interest to other areas of mathematics. 
\medskip 

\noindent 
\textit{A $p$-local geometric Poincar\'{e}-Birkhoff-Witt Theorem}. 
Combining Theorems~\ref{coHfib} and~\ref{HM} allows for a 
complete decomposition of the loops on a co-$H$-space into 
a product of functorially indecomposable spaces. That is, 
Theorem~\ref{coHfib} states that for a simply-connected 
co-$H$-space $Y$ there is a homotopy decomposition 
\[\Omega Y\simeq A^{\min}(Y)\times 
      \Omega(\bigvee_{n=2}^{\infty} Q_{n} B(Y))\] 
where $A^{\min}(Y)$ is functorially indecomposable. Since each space 
$Q_{n} B(Y)$ is a co-$H$-space, Theorem~\ref{HM} implies that there 
is a homotopy decomposition 
\[\Omega(\bigvee_{n=2}^{\infty} Q_{n} B(Y))\simeq 
       \prod_{\alpha\in\mathcal{I}}\Omega M((Q_{n}B(Y),\alpha_{n})_{n=2}^{\infty})\] 
where each space $M((Q_{n}B(Y),\alpha_{n})_{n=2}^{\infty})$ is a 
simply-connected co-$H$-space. 
The homotopy decomposition in Theorem~\ref{coHfib} can now be applied 
to each of the factors $\Omega M((Q_{n}B(Y),\alpha_{n})_{n=2}^{\infty})$ 
to produce an $A^{\min}$ that is functorially indecomposable and a complementary 
factor which is the loops on a wedge of simply-connected co-$H$-spaces. 
Iterating, we obtain a decomposition of $\Omega Y$ as a product of $A^{\min}$'s. 

\begin{thm} 
   \label{Amincomplete} 
   Let $Y$ be a simply-connected co-$H$-space. Then there is a 
   functorial homotopy decomposition 
   \[\Omega Y\simeq\prod_{\gamma\in\mathcal{J}} A^{\min}(Y_{\gamma})\] 
   for some index set $\mathcal{J}$, where each $Y_{\gamma}$ is a 
   simply-connected co-$H$-space and each factor $A^{\min}(Y_{\gamma})$ 
   is functorially indecomposable.~$\qqed$ 
\end{thm} 

Theorem~\ref{Amincomplete} is related to the Poincar\'{e}-Birkhoff-Witt 
Theorem. To explain how, let $V$ be a module over a field of characteristic $p$, 
and let $T(V)$ be the tensor algebra generated by $V$. This is made 
into a Hopf algebra by declaring that the generators are primitive and 
extending multiplicatively. There is a canonical isomorphism 
$T(V)\cong UL(V)$, where $L(V)$ is the free Lie algebra on~$V$ 
and $UL(V)$ is its universal enveloping algebra. Let $L_{n}(V)$ be 
the homogeneous component of the free Lie algeba $L(V)$ of 
tensor length $n$. The Poincar\'{e}-Birkhoff-Witt Theorem states 
that rationally there is a functorial coalgebra isomorphim 
$T(V)\cong\otimes_{n=1}^{\infty} S(L_{n}(V))$, where 
$S(\ )$ is the free symmetric algebra functor. Note that each 
$S(L_{n}(V))$ further decomposes as a product of exterior 
algebras and polynomial algebras on a single generator, 
with the generators in one-to-one correspondence with 
the module generators of~$L_{n}(V)$. This can be geometrically 
realized. If $Y$ is a simply-connected co-$H$-space such that 
$\hlgy{\Omega Y;\mathbb{Z}/p\mathbb{Z}}\cong T(V)$, then 
there is a rational homotopy equivalence 
$\Omega Y\simeq\prod_{n=1}^{\infty} S_{n}$ 
where $S_{n}$ is a product of odd dimensional spheres and 
the loops on odd dimensional spheres, with the property that 
$\hlgy{S_{n};\mathbb{Q}}\cong S(L_{n}(V))$. Theorem~\ref{Amincomplete} 
is a $p$-local analogue, in the sense that it produces a decomposition 
of $\Omega Y$ into a product of functorially indecomposable pieces, 
each of which geometrically realizes a functorially indecomposable 
factor of $T(\Sigma^{-1}\rhlgy{Y})$.  
\medskip 

\noindent 
\textit{Decompositions of Lie powers}. 
As above, let $V$ be a module over a field of characteristic $p$, regard 
$T(V)$ as $UL\langle V\rangle$, and let $L_{n}(V)$ be the homogeneous 
component of the free Lie algeba $L(V)$ of tensor length $n$. The 
module $L_{n}(V)$ is called the \emph{$n^{th}$ free Lie power} of $V$. 
Decompositions of Lie powers $L_{n}(V)$ over 
the general linear group $GL(V)$ is a subject of considerable recent 
activity in modular representation theory (see, for example, \cite{BS,ES}).  

Now suppose that $Y$ is a simply-connected co-$H$-space 
and $V=\Sigma^{-1}\rhlgy{Y}$. Then $\hlgy{\Omega Y}\cong T(V)$.  
In general, any functorial homotopy decomposition of $\Omega Y$ determines 
a functorial coalgebra decomposition of $T(V)$, which in turn determines 
a module decomposition of $L_{n}(V)$ over the general linear group 
for every $n$. Thus studying homotopy decompositions of $\Omega Y$ 
gives information about how Lie powers decompose.  

The method developed in~\cite{SW1}, introducing a functorial 
Poincar\'{e}-Birkhoff-Witt theorem, gives a fundamental connection 
between the homotopy theory of loops on co-$H$-spaces and 
the modular representation theory of Lie powers. It was shown 
in~\cite{LLW} that the modular representation theory of Lie powers 
over the general linear group is tightly related to functorial coalgebra 
decompositions of tensor algebras. Through this connection, one 
can investigate topological applications of new developments in 
representation theory. Conversely, topological methods such as 
Hopf invariants and techniques in Hopf algebras provide tools 
different from traditional methods in representation theory for 
studying Lie powers. Hopf invariants were obtained in geometry 
from the suspension splittings of loops on co-$H$-spaces. The 
homological behavior of Hopf invariants gives a family of natural 
coalgebra maps on tensor algebras involving certain important 
combinatorial information on shuffles~\cite{SW1}. By considering 
Hopf invariants together with techniques in Hopf algebras, \cite{LLW}  
generalized some important recent results on the representation theory 
of Lie powers given by Bryant-Schocker~\cite{BS}. 

For example, in~\cite{LLW} it was shown that the sub-Hopf algebra 
$B(V)$ of $T(V)$ generated by the set 
$\{L_{n}(V)\mid\mbox{$n$ is not a power of $p$}\}$ is a 
functorial coalgebra summand of $T(V)$. This was used to construct 
an explicit decomposition~\cite[6.3]{LLW} of $L_{m}(V)$ when $m$ 
is not a power of $p$. For our purposes, observe that $B(V)$ is a 
coalgebra-split sub-Hopf algebra of $T(V)$. So Theorems~\ref{geomreal} 
and~\ref{geomrealfib} imply that the Hopf-algebra map 
\(\namedright{B(V)}{}{T(V)}\) 
can be geometrically realized as a loop map 
\(\namedright{\Omega(\bigvee_{n=1}^{\infty}\bar{Q}_{n} B(Y))} 
     {}{\Omega Y}\) 
for some simply-connected co-$H$-space $Y$. 
The advantage of having a geometric realization is that it 
is stronger than simply having an algebraic decomposition. 
The topology of the decomposition may imply additional 
algebraic information beyond that used in~\cite{LLW}, which 
may give further insight into how Lie powers decompose. We 
leave specific applications of this to later work. 
\medskip 

\noindent 
\textit{Homotopy types of generalized moment-angle complexes}.  
Let $X_{1},\ldots,X_{m}$ be simply-connected, pointed $CW$-complexes of 
finite type. For $1\leq k\leq m$, define the space~$T^{m}_{k}$ by 
\[T^{m}_{k}=\{(x_{1},\ldots,x_{m})\in\prod_{i=1}^{m} X_{i}\mid 
       \mbox{at least $k$ of $x_{1},\ldots,x_{m}$ are the basepoint}\}.\] 
Inclusion into the product gives a map 
\(\namedright{T^{m}_{k}}{}{\prod_{i=1}^{m} X_{i}}\). 
Define the space $F^{m}_{k}$ by the homotopy fibration 
\[\nameddright{F^{m}_{k}}{}{T^{m}_{k}}{}{\prod_{i=1}^{m} X_{i}}.\] 
In~\cite{P} it was shown that there is a homotopy equivalence 
\begin{equation}
  \label{Fmkdecomp}
   F^{m}_{k}\simeq\bigvee_{j=m-k+1}^{m}\left(
     \bigvee_{1\leq i_{1}<\cdots<i_{j}\leq m}\binom{j-1}{m-k}
     \Sigma^{m-k}\Omega X_{i_{1}}\wedge\cdots\wedge\Omega X_{i_{j}}\right).
\end{equation} 
In particular, Theorem~\ref{Porter} is the special case when $k=m-1$. 

The spaces $F^{m}_{k}$ have received a great deal of attention 
lately as they are special cases of the generalized moment-angle 
complexes defined in~\cite{BBCG}. The classical moment-angle 
complexes are given by the special case when each $X_{i}=\cpinf$;  
they are fundamental objects in toric topology (see~\cite{BP,DJ}). 
Moreover, classical moment-angle complexes can be identified with 
complements of complex coordinate subspaces~\cite{BP}, which are 
fundamental objects in combinatorics (see, for example,~\cite{Bj}),  
and a major problem is to determine their homotopy type. It is 
therefore natural to also try to determine the homotopy type of 
generalized moment-angle complexes. Progress in this direction has 
been made in~\cite{GT1,GT2}.

In the special case of the spaces $F^{m}_{k}$, the decomposition 
in~(\ref{Fmkdecomp}) goes a long way towards determining the homotopy 
type. This can be refined considerably when each $X_{i}$ is a suspension,  
$X_{i}=\Sigma\overline{X}_{i}$, by iterating James' decomposition 
$\Sigma\Omega\Sigma X\simeq\bigvee_{n=1}^{\infty} X^{(n)}$. 
Doing so, one obtains a homotopy decomposition of 
$F^{m}_{k}$ as a large wedge of spaces of the form 
$\Sigma^{m-k}\overline{X}_{j_{1}}^{(n_{j_{1}})}\wedge\cdots\wedge 
         \overline{X}_{j_{l}}^{(n_{j_{l}})}$ 
where $1\leq j_{1}<\cdots <j_{l}\leq m$. Moreover, one obtains 
a decomposition of $\Omega F^{m}_{k}$ as the loops on a large 
wedge of suspensions. The Hilton-Milnor Theorem can now be 
applied to decompose further. 

All of this can now be generalized to the case of $F^{m}_{k}$ 
when each $X_{i}$ is a simply-connected co-$H$-space. The 
iteration of James' decomposition is replaced by the decomposition 
in Proposition~\ref{GTWprop} and the Hilton-Milnor Theorem is 
then replaced by its generalization in Theorem~\ref{HM}. 
\medskip 

\noindent 
\textit{Decompositions of the ring of quasi-symmetric functions}.   
The Hopf-algebra of non-symmetric functions $\NSymm$ is defined 
as the tensor algebra $T(z_{1},z_{2},\ldots)$, where $\vert z_{i}\vert=2i$, 
the coproduct is given by $\Delta(z_{n})=\Sigma_{s+t=n}\, z_{s}\otimes z_{t}$,  
and the antiautomorphism is given by 
$\chi(z_{n})=\Sigma_{\alpha_{1}+\cdots+\alpha_{m}=n}\, z_{\alpha_{1}}\cdots 
           z_{\alpha_{m}}$. 
The Hopf algebra of quasi-symmetric functions $\QSymm$ is defined 
as the Hopf algebra dual of $\NSymm$.  
In~\cite{BR}, Baker and Richter observed that there is an integral 
Hopf-algebra isomorphism 
$\cohlgy{\Omega\Sigma\cpinf;\mathbb{Z}}\cong\QSymm$. 
They then used topological properties of $\Omega\Sigma\cpinf$ to 
prove algebraic properties of $\QSymm$. 

In general, a $p$-local homotopy decomposition 
$\Omega\Sigma\cpinf\simeq\prod_{\alpha} A_{\alpha}$ 
for some spaces~$A_{\alpha}$ implies that there is a $p$-local 
algebra decomposition 
$\rcohlgy{\Omega\Sigma\cpinf;\plocal}\cong 
     \otimes_{\alpha}\rcohlgy{A_{\alpha};\plocal}$. 
Therefore we obtain a $p$-local algebra decomposition 
$\QSymm\cong\otimes_{\alpha}\rcohlgy{A_{\alpha};\plocal}$. 
Baker and Richter gave the following example. Recall that 
$\cohlgy{\cpinf;\mathbb{Z}}\cong\mathbb{Z}[x]$, where $x$ is 
in degree~$2$. In~\cite{MNT} it was shown that that there is 
a $p$-local homotopy decomposition 
$\Sigma\cpinf\simeq\bigvee_{i=1}^{p-1} A_{i}$ 
where $\cohlgy{A_{i}}$ consists of those elements in 
$\Sigma\cohlgy{\cpinf}$ in degrees of the form $2i+1+2k(p-1)$ 
for some $k\geq 0$. We obtain a homotopy decomposition 
\begin{equation} 
  \label{BRdecomp} 
  \Omega\Sigma\cpinf\simeq\Omega(\bigvee_{i=1}^{p-1} A_{i}). 
\end{equation}  
Note that each $A_{i}$ is a simply-connected co-$H$-space as it is 
a retract of $\Sigma\cpinf$. Baker and Richter used this to further 
decompose $\Omega\Sigma\cpinf$ by anticipating the 
generalization of the Hilton-Milnor Theorem in Theorem~\ref{HM}. 

We give a different $p$-local homotopy decomposition of $\Omega\Sigma\cpinf$ 
which is finer than Baker and Richter's, and which therefore implies 
a correspondingly finer $p$-local algebra decomposition of $\QSymm$. 
By~\cite{Ga}, for any simply-connected space~$X$ there is a homotopy 
fibration 
\(\nameddright{\Sigma\Omega X\wedge\Omega X}{}{\Sigma\Omega X} 
       {ev}{X}\) 
which splits after looping as 
$\Omega\Sigma\Omega X\simeq 
      \Omega X\times\Omega(\Sigma\Omega X\wedge\Omega X)$. 
In our case, let $B\cpinf$ be the Eilenberg-MacLane space $K(\mathbb{Z},3)$, 
so $\Omega B\cpinf\simeq\cpinf$. Then we obtain a homotopy fibration 
\[\nameddright{\Sigma\cpinf\wedge\cpinf}{}{\Sigma\cpinf}{ev}{B\cpinf}\] 
and a homotopy decomposition 
\[\Omega\Sigma\cpinf\simeq\cpinf\times\Omega(\Sigma\cpinf\wedge\cpinf).\] 

Consider three decompositions of the space $\Sigma\cpinf\wedge\cpinf$.  
First, applying the decomposition from~\cite{MNT} on the left factor we obtain 
\[\Sigma\cpinf\wedge\cpinf\simeq 
       (\bigvee_{i=1}^{p-1} A_{i})\wedge\cpinf\simeq 
       \bigvee_{i=1}^{p-1} (A_{i}\wedge\cpinf).\] 
Second, moving the suspension to the right wedge summand gives a 
similar homotopy decomposition 
\[\cpinf\wedge\Sigma\cpinf\simeq 
         \cpinf\wedge(\bigvee_{i=1}^{p-1} A_{i})\simeq 
         \bigvee_{i=1}^{p-1}(\cpinf\wedge A_{i}).\] 
Third, let   
\(T\colon\namedright{\cpinf\wedge\cpinf}{}{\cpinf\wedge\cpinf}\) 
be the map interchanging factors. Define self-maps of 
$\Sigma\cpinf\wedge\cpinf$ by $e_{1}=(1+\Sigma T)/2$ and 
$e_{2}=(1-\Sigma T)/2$. Observe that $(e_{1})_{\ast}$ and $(e_{2})_{\ast}$ 
are idempotents, $(e_{1})_{\ast}\circ (e_{2})_{\ast}=0$, and 
$(e_{1})_{\ast}+(e_{2})_{\ast}=1$. Therefore, if $E_{1}$ and $E_{2}$ 
are the mapping telescops of $e_{1}$ and $e_{2}$ respectively, then 
adding gives a map 
\(e\colon\namedright{\Sigma\cpinf\wedge\cpinf}{}{E_{1}\vee E_{2}}\) 
which is a homology isomorphism. Since $\Sigma\cpinf\wedge\cpinf$ 
is simply-connected, Whitehead's theorem implies that $e$ is a 
homotopy equivalence, giving a decomposition 
\[\Sigma\cpinf\wedge\cpinf\simeq\bigvee E_{1}\vee E_{2}.\] 
 
The homotopy theoretic Krull-Schmidt theorem in~\cite{Gr} 
implies that if $Y$ is a simply-connected co-$H$-space and 
there are decompositions 
\(f\colon\namedright{Y}{\simeq}{\bigvee_{i=1}^{n} A_{i}}\)  
and 
\(g\colon\namedright{Y}{\simeq}{\bigvee_{j=1}^{m} B_{j}}\) 
then $g$ determines a decomposition of each space $A_{i}$ 
into a wedge of $j$ summands, and similarly for $f$ with 
respect to each $B_{j}$. In our case, the first two decompositions 
of $\Sigma\cpinf\wedge\cpinf$ above combine to produce a 
homotopy decomposition 
\[\Sigma\cpinf\wedge\cpinf\simeq\bigvee_{i,j=1}^{p-1} A_{i,j}\] 
where $\rcohlgy{A_{i,j}}$ consists of those elements in 
$\cohlgy{\Sigma\cpinf\wedge\cpinf}$ in bidegrees of the form 
$(2i+1+2k(p-1), 2i+2l(p-1))$ for $k,l\geq 0$. Combining this 
with the third decomposition above gives a refined decomposition 
\[\Sigma\cpinf\wedge\cpinf\simeq\bigvee_{i,j=1}^{p-1} 
       (A_{i,j}^{+}\vee A_{i,j}^{-})\]   
where $\cohlgy{A_{i,j}^{+}}$ consists of those elements in 
$\cohlgy{A_{i,j}}$ which are also symmetric when considered 
as elements of $\cohlgy{\Sigma\cpinf\wedge\cpinf}$ and 
$\cohlgy{A_{i,j}^{-}}$ is the corresponding complement. Hence 
there is a homotopy decomposition 
\begin{equation} 
  \label{refinedQS} 
  \Omega\Sigma\cpinf\simeq\cpinf\times 
      \Omega\left(\bigvee_{i,j=1}^{p-1} (A_{i,j}^{+}\vee A_{i,j}^{-})\right). 
\end{equation}  

Now to give a $p$-local decomposition of $\QSymm$, we can 
start from the homotopy decomposition in~(\ref{refinedQS}) 
involving $p(p-1)$ wedge summands rather than the decomposition 
in~(\ref{BRdecomp}) involving just $p-1$ summands. As before, 
Theorem~\ref{HM} can be applied to decompose 
$\Omega(\bigvee_{i,j=1}^{p-1} (A_{i,j}^{+}\vee A_{i,j}^{-}))$ 
into a product of looped co-$H$-spaces. If desired, each 
factor can be even further decomposed using 
Theorem~\ref{Amincomplete}.

\end{document}